\documentclass[12pt,english]{amsart}
\usepackage{amsmath}
\usepackage{libertine}
\usepackage[libertine]{newtxmath}
\usepackage{mathrsfs}

\usepackage[T1]{fontenc}
\usepackage{hyperref}
\usepackage{url}
\usepackage{amssymb,latexsym}
\usepackage{cite,amsmath,amstext,amsbsy,bm} 
\usepackage{paralist}
\usepackage{datetime} 
\usepackage{colortbl}
\usepackage[dvipsnames]{xcolor}
\usepackage{mathtools}
\usepackage{euscript}
\usepackage[all]{xy}

\usepackage{graphicx} 
\usepackage[a4paper,text={150mm,257mm},centering]{geometry} 
\usepackage{comment}  
\usepackage{cite}  
\usepackage{thmtools}
\usepackage{enumitem}
\usepackage{letltxmacro} 
\usepackage{nameref}
\usepackage{cleveref}
\usepackage{calc}
\usepackage{interval}
\usepackage{manfnt}
\usepackage{tikz-cd}  
\usepackage[utf8]{inputenc}
\usepackage{newunicodechar} 
\usepackage[hyperpageref]{backref}

\usepackage{csquotes}

 \newcommand{\lowerromannumeral}[1]{\romannumeral#1\relax}
 
\theoremstyle{plain}
\newtheorem{thmx}{Theorem}
\renewcommand{\thethmx}{\Alph{thmx}} 
\newtheorem{thm}{Theorem}[section]  
\newtheorem{lem}[thm]{Lemma}

\newtheorem{prop}[thm]{Proposition}

\newtheorem{claim}[thm]{Claim}

\theoremstyle{definition}
\newtheorem{dfn}[thm]{Definition}

\theoremstyle{remark}
\newtheorem{rem}[thm]{Remark}

\numberwithin{equation}{subsection}  

\theoremstyle{plain}
\newlist{thmlist}{enumerate}{1}
\setlist[thmlist]{wide = 0pt, labelwidth = 2em, labelsep*=0em, itemindent = 0pt, leftmargin = \dimexpr\labelwidth + \labelsep\relax, noitemsep,topsep = 1ex, font=\normalfont, label=(\roman*), ref=\thethm.(\roman{thmlisti})}

\addtotheorempostheadhook[thm]{\crefalias{thmlisti}{thm}}

\addtotheorempostheadhook[assumpsion]{\crefalias{thmlisti}{assumption}}

\addtotheorempostheadhook[prop]{\crefalias{thmlisti}{prop}}

\addtotheorempostheadhook[dfn]{\crefalias{thmlisti}{dfn}}

\addtotheorempostheadhook[lem]{\crefalias{thmlisti}{lem}}
\addtotheorempostheadhook[main]{\crefalias{thmlisti}{main}}

\addtotheorempostheadhook[rem]{\crefalias{thmlisti}{rem}}

\newlist{thmenum}{enumerate}{1} % also creates a counter called 'propenumi'
\setlist[thmenum]{wide = 0pt, labelwidth = 2em, labelsep*=0em, itemindent = 0pt, leftmargin = \dimexpr\labelwidth + \labelsep\relax, noitemsep,topsep = 1ex, font=\normalfont, label=(\roman*), ref=\thethmx.(\roman{thmenumi})}%{label=\alph*), ref=\thethmx~(\alph*)}
\crefalias{thmenumi}{thmx}

\makeatletter
\newsavebox{\@brx}
\newcommand{\llangle}[1][]{\savebox{\@brx}{\(\m@th{#1\langle}\)}%
	\mathopen{\copy\@brx\kern-0.5\wd\@brx\usebox{\@brx}}}
\newcommand{\rrangle}[1][]{\savebox{\@brx}{\(\m@th{#1\rangle}\)}%
	\mathclose{\copy\@brx\kern-0.5\wd\@brx\usebox{\@brx}}}
\makeatother

\crefname{lem}{Lemma}{Lemmas}
\crefname{thm}{Theorem}{Theorems}
\crefname{proposition}{Proposition}{Propositions}
\crefname{dfn}{Definition}{Definitions}
\crefname{rem}{Remark}{Remarks}
\crefname{cor}{Corollary}{Corollaries}
\crefname{corx}{Corollary}{Corollaries}
\crefname{problem}{Problem}{Problems}
\crefname{thmx}{Theorem}{Theorems}
\crefname{claim}{Claim}{Claims}
\crefname{assumption}{Assumption}{Assumptions}
\crefname{main}{Main Theorem}{Main Theorems}

\makeatletter
\newcommand*{\rom}[1]{\expandafter\@slowromancap\romannumeral #1@}
\makeatother

\makeatletter     %Replace the Section ...  by the symbol \S ...  in Cref
\newcommand{\crefnames}[3]{%
	\@for\next:=#1\do{%
		\expandafter\crefname\expandafter{\next}{#2}{#3}%
	}%
}
\makeatother

\crefnames{part,chapter,section}{\S}{\S\S}

\newcommand{\sE}{\mathscr{E}}
\newcommand{\sF}{\mathscr{F}}

\newcommand{\sL}{\mathscr{L}}

\newcommand{\sP}{\mathscr{P}}

\newcommand{\cO}{\mathcal O}

\newcommand{\bC}{\mathbb{C}}
\newcommand{\bD}{\mathbb{D}}

\newcommand{\bP}{\mathbb{P}}

 \def\d{\partial} 
\def\hess{\sqrt{-1}\partial\overline{\partial}}

\def\sn{\sqrt{-1}}

\makeatletter
\hypersetup{
	pdfauthor={\authors},
	pdftitle={\@title},
	pdfsubject={\@subjclass},
	pdfkeywords={\@keywords},
	pdfstartview={Fit},
	pdfpagelayout={TwoColumnRight},
	pdfpagemode={UseOutlines},
	bookmarks,
	colorlinks,
	linkcolor=Sepia,
	citecolor=OliveGreen,
	urlcolor=WildStrawberry}
\makeatother

\begin{document} 
	
	\title[Big Picard  type   theorem]{Big Picard    theorem  for moduli \\ spaces of polarized manifolds}

	\author{Ya Deng} 
	%\thanks{This work is supported by IH\'ES}
 \address{
		Institut des Hautes Études Scientifiques,
		Universit\'e  Paris-Saclay, 35 route de Chartres, 91440, Bures-sur-Yvette, France}
	\email{deng@ihes.fr}
	
	\urladdr{https://www.ihes.fr/~deng} 
	\date{\today} 
	\begin{abstract} 
		Consider a smooth projective family of complex polarized  manifolds with semi-ample canonical sheaf over a quasi-projective manifold $V$. When the associated moduli map $V\to P_h$ from the base to  coarse moduli space is quasi-finite, we prove that     the generalized big Picard   theorem holds for the base manifold $V$: for any projective compactification $Y$ of $V$, any holomorphic map $f:\Delta-\{0\}\to V$ from the punctured unit disk to $V$ extends to a holomorphic map  of the unit disk $\Delta$ into $Y$.  This result generalizes our previous work on the Brody hyperbolicity of $V$ (\emph{i.e.} there are no entire curves on $V$), as well as  a more recent work by Lu-Sun-Zuo on the Borel hyperbolicity of $V$ (\emph{i.e.} any holomorphic map from a quasi-projective variety to $V$ is algebraic). We also obtain   generalized big Picard   theorem for bases of log Calabi-Yau families.
\end{abstract} 
	\maketitle
	
\tableofcontents
\section{Introduction} 
The classical big Picard theorem says that any holomorphic map  from
the punctured disk $\Delta^*$  into $\bP^1$
which omits three points can be extended to a
holomorphic map $\Delta\to \bP^1$, where $\Delta$ denotes the unit disk. %By the work of Kobayashi et  al., we know that big Picard theorem is related to hyperbolicity. Indeed, 
Therefore, we say the (generalized) big Picard theorem holds for  a quasi-projective variety $V$     if for some (thus any) projective compactification $Y$ of $V$, any holomorphic map $f:\Delta^*\to V$ extends to a holomorphic map $\bar{f}:\Delta\to X$.  This property is interesting for it implies the \emph{Borel hyperbolicity}\footnote{The notion of Borel hyperbolicity was first introduced by Javanpeykar-Kucharczyk in \cite{JK18}.} of $V$: any holomorphic map from a quasi-projective variety to $V$ is necessarily \emph{algebraic}. 
A natural question is to find algebraic varieties satisfying the big Picard theorem.  By the fundamental work of Kobayashi (see \cite[Theorem 6.3.6]{Kob98}), big Picard theorem holds for the quasi-projective manifold $V$ which admits a projective compactification $Y$ such that $V$ is \emph{hyperbolically embedded into} $Y$ (see \cite[Chapter 3. \S3]{Kob98} for the definition). This gives us an important criteria for varieties satisfying the big Picard theorem. By the work of Fujimoto \cite{Fuj72} and Green \cite{Gre77},  complements of $2n+1$ general hyperplanes in $\bP^n$ are hyperbolically embedded into $\bP^n$. More recently Brotbek and the author \cite{BD19} proved that the complement  of a general hypersurface  in $\bP^n$ of high degree is also hyperbolically embedded into $\bP^n$. By A. Borel \cite{Bor72} and Kobayashi-Ochiai \cite{KO71}, the quotients of bounded symmetric domains by torsion free arithmetic groups are hyperbolically embedded into its Baily-Borel compactification. Hence these results provide  examples of quasi-projective manifolds satisfying the big Picard theorem.

However, quasi-projective manifold $V$ being hyperbolically embedded into some projective compactification $Y$ is \emph{minimal} in the sense that, for any birational modification $\mu:\tilde{Y}\to Y$ by blowing-up the boundary $Y\setminus V$, $\mu^{-1}(V)\simeq V$ is no more hyperbolically embedded into $\tilde{Y}$, while the  big Picard theorem does not depends on its compactification. The first result in this paper is to establish a more flexible criteria for big Picard theorem. For our definition of Finsler metric (which is slightly different from the usual one in the literature), see \cref{def:Finsler}. 
\begin{thmx}[=\cref{thm:Big Picard}]\label{main1}
Let $X$ be a projective manifold   and let $D$ be a simple normal crossing divisor on $X$. Let $h$ be a (possibly degenerate) Finsler metric of $T_X(-\log D)$.  Assume that $f:\Delta^*\to X\setminus D$  is a  holomorphic map from the punctured unit disk $\Delta^*$ to $X\setminus D$  such that $  |f'(t)|_h^2\not\equiv 0$, and 
\begin{align*} 
 \hess \log |f'(t)|_h^2\geq f^*\omega
\end{align*}
for some smooth K\"ahler metric $\omega$ on $X$. Then $f$ extends to a holomorphic map $\bar{f}:\Delta\to X$ of the unit disk into $X$.
\end{thmx}
\cref{main1} is inspired by   \emph{fundamental vanishing theorem} for  jet differentials vanishing on some ample divisor by Siu-Yeung \cite{S-Y97} and Demailly \cite[\S 4]{Dem97b}, and its proof is   mainly based on a \emph{logarithmic derivative lemma} by Noguchi \cite{Nog81}.

The motivation of \cref{main1} is to study the hyperbolicity of moduli spaces of polarized manifolds with semi-ample canonical bundle. By the fundamental work of Viehweg-Zuo \cite{VZ02,VZ03}, and the recent development by Popa et al. \cite{PS17,PTW18} and \cite{Den18,Den18b,Den19,LSZ19}, a special Higgs bundle, the so-called Viehweg-Zuo Higgs bundle  in \cref{def:VZ} below, turns out to be a powerful technique in studying hyperbolicity problems.  For a quasi-projective manifold $V$  equipped with a Viehweg-Zuo Higgs bundle, in  \cite{Den18,Den18b} we prove  that $V$ can be equipped with  a \emph{generically positively definite} Finsler metric whose holomorphic sectional curvature is bounded from above by a negative constant. In particular, we prove that $V$ is always \emph{pseudo Kobayashi hyperbolic}. The second aim of this article is to give a new curvature estimate for that Finsler metric on $V$.
\begin{thmx}[=\cref{thm:uniform}]\label{main2}
	Let $X$ be a projective manifold  equipped with a smooth K\"ahler metric  $\omega$ and let $D$ be a simple normal crossing divisor on $X$. Assume that there is a Viehweg-Zuo Higgs bundle over $(X,D)$. Then there are a positive constant $\delta$ and a Finsler metric $h$ on $T_X(-\log D)$ which is positively definite on a dense Zariski open set $V^\circ$ of $V:=X\setminus D$,   such that   for any holomorphic map $\gamma:C\to V$  from  any open subset $C$ of $ \bC$ to $V$ with $\gamma(C)\cap V^\circ\neq \varnothing$, one has
	$$
	\hess \log |\gamma'|_h^2\geq \delta\gamma^*\omega.
	$$ 
	In particular, by \cref{main1},	   for any holomorphic map   $f:\Delta^*\to X\setminus D$, with $f(\Delta^*)\cap V^\circ\neq \varnothing$,  it extends to a holomorphic map $\bar{f}:\Delta\to X$.
\end{thmx} 
By the work \cite{VZ02,VZ03,PTW18},  Viehweg-Zuo Higgs bundles exist  on bases of maximally varying, smooth family of projective manifolds with semi-ample canonical bundle. Combining their results and \cref{main2}, we prove the big Picard theorem for moduli of polarized manifolds with semi-ample canonical bundle.
\begin{thmx}\label{main}
	Consider the moduli functor $\mathscr{P}_h$ of polarized manifolds with semi-ample canonical sheaf  introduced by Viehweg \cite[\S 7.6]{Vie95}, where $h$ is the  Hilbert polynomial associated to the polarization.  Assume that for some quasi-projective manifold $V$ over which there exists
	a smooth polarized family $(f_U:U\to V,\sL)\in \mathscr{P}_h(V)$ such that the induced moduli map $\varphi_U:V\to P_h$ 
	is quasi-finite. Let $Y$ be an arbitrary  projective compactification  of $V$. Then any holomorphic map $\gamma:\Delta^*\to V$ from the punctured unit disk $\Delta^*$ to $V$ extends to  a holomorphic map from the unit disk $\Delta$ to $Y$.
\end{thmx}

Under the same assumption as \cref{main2}, we have already in \cite{Den18b}   proved the Brody hyperbolicity of $V$: there exists no entire curves $\gamma:\bC\to V$. Based on the  infinitesimal Torelli-type theorem proven in \cite[Theorem C]{Den18b} (see  \cref{Deng} below), more recently,  Lu-Sun-Zuo \cite{LSZ19} proved the Borel hyperbolicity of $V$: any holomorphic map from a quasi-projective variety to $V$ is algebraic. The use of Nevannlina theory in this article is inspired by their work, although our methods are different from theirs (see \cref{rem:LSZ}).

Finally, let us mention that in \cite{Den19}, we construct Viehweg-Zuo Higgs bundles over bases of maximally varying, log smooth families of Calabi-Yau families (see \cref{def:log smooth} for the definition of log smooth family). Applying \cref{main1,main2} to this result, we also obtain  big Picard theorem  for these base manifolds.
\begin{thmx}[=\cref{thm:CY}] \label{main:CY}
	Let $f^\circ:(X^\circ,D^\circ)\to Y^\circ$ be a log smooth family  over a quasi-projective manifold $Y^\circ$. Assume that  each fiber $(X_y,D_y):=(f^\circ)^{-1}(y)$ of $f^\circ$ is a klt pair, and  $K_{X_y}+D_y\equiv_{\mathbb{Q}}0$. Assume that the logarithmic Kodaira-Spencer map 
	$$
	T_{Y^\circ,y}\rightarrow  H^1 \big(X_y,T_{X_y}(-\log D_y)\big)
	$$
	is injective for any $y\in Y^\circ$. Then for any projective compactification $Y$ of the base $Y^\circ$,  any holomorphic map $\gamma:\Delta^*\to Y^\circ$ extends to  a holomorphic map from  $\Delta$ to $Y$. 
\end{thmx}

\subsection{Acknowledgments.}  I would like to thank     Professors Jean-Pierre Demailly and Emmanuel Ullmo for their encouragements and supports.  This work is supported by Institut des Hautes \'Etudes Scientifiques.

\section{A differential geometric criteria for big Picard  theorem}
In the similar vein as the fundamental vanishing theorem  for jet differentials vanishing on some ample divisor by Siu-Yeung \cite{S-Y97} and Demailly \cite{Dem97b}, in this section we will establish a differential geometric criteria for big Picard type theorem via the logarithmic derivative lemma by Noguchi \cite{Nog81}. 
\subsection{Preliminary in Nevannlina theory}
Let $\mathbb{D}^*:=\{t\in \bC\mid |t|>1 \}$, and $\mathbb{D}:=\mathbb{D}^*\cup \infty$. Then via the map $z\mapsto \frac{1}{z}$, $\mathbb{D}^*$ is isomorphic to the punctured unit disk $\Delta^*$ and $\mathbb{D}$ is isomorphic to the unit disk $\Delta$. Therefore, for any holomorphic map $f$ from the punctured disk $\Delta^*$ into a projective variety $Y$, $f$ extends to the origin if and only if $f(\frac{1}{z}):\mathbb{D}^*\to Y$ extends to the infinity.

Let $(X,\omega)$ be a compact K\"ahler manifold, and $\gamma:\mathbb{D}^*\to X$ be a holomorphic map. Fix any $r_0>1$. Write $\mathbb{D}_r:=\{z\in \bC\mid r_0<|z|<r \}$.   The \emph{order function} is defined  by
$$
T_{\gamma,\omega}(r):= \int_{r_0}^{r}\frac{d\tau}{\tau}\int_{\bD_\tau}\gamma^*\omega.
$$ 
As is well-known, the asymptotic behavior of $T_{\gamma,\omega}(r)$ as $r\to \infty$ characterizes whether $\gamma$ can be extended over the $\infty$  (see \emph{e.g.} \cite[2.11. Cas   \guillemotleft local \guillemotright]{Dem97b} or \cite[Remark 4.7.4.(\lowerromannumeral{2})]{NW14}).  
\begin{lem}\label{lem:criteria}
$T_{\gamma,\omega}(r)=O(\log r)$ if and only if $\gamma$ is extended holomorphically over  $\infty$. \qed
\end{lem}
The following lemma is well-known to experts (see \emph{e.g.} \cite[Lemme 1.6]{Dem97b}).
\begin{lem}\label{lem:decrease}
	Let $ X$ be a projective manifold equipped with a hermitian metric $\omega$ and let $u:X\to \mathbb{P}^1$ be a rational function. Then for any holomorphic map $\gamma:\mathbb{D}^*\to X$, one has
	$$
	T_{u\circ\gamma,\omega_{FS}}(r)\leq CT_{\gamma,\omega}(r)+O(1)
	$$
	where $\omega_{FS}$ is the Fubini-Study metric for $\mathbb{P}^1$. \qed
\end{lem}

The following logarithmic derivative lemma by Noguchi is crucial in the proof of \cref{main1}.
\begin{lem}[\protecting{\cite[Lemma2.12]{Nog81},\cite[3.4. Cas local]{Dem97b}}]\label{lem:Noguchi}
	Let $u:\mathbb{D}^*\to \mathbb{P}^1$ be any meromorphic function. Then we have
	$$
	\frac{1}{2\pi}\int_{0}^{2\pi}\log ^+|(\log u)'(re^{i\theta})|d\theta \leq C(\log^+ T_{u,\omega_{FS}}(r)+\log r)+O(1)   \quad \lVert,
	$$
	for some constant $C>0$ which does not depend on $r$. 
	Here   the symbol \(\lVert\) means that	the  inequality holds outside a Borel subset of \((r_0,+\infty )\) of finite Lebesgue measure.  \qed
\end{lem}
We need the lemma by E. Borel.
\begin{lem}[\protecting{\cite[Lemma 1.2.1]{NW14}}]\label{lem:Borel}
	Let $\phi(r)\geq 0 (r\geq r_0\geq 0)$ be a monotone increasing
	function. For every $\delta>0$, 
	$$
	\frac{d}{d r}\phi(r)\leq \phi(r)^{1+\delta}\quad \lVert.
	$$  \qed
\end{lem}
We recall two useful formulas (the second one is the well-known Jensen formula).
\begin{lem}
	Write $\log^+x:={\rm max}(\log x ,0)$. 
	\begin{align}\label{eq:concave}
 \log^+(\sum_{i=1}^{N}x_i)\leq \sum_{i=1}^{N}\log^+x_i+ \log N\quad\mbox{ for }x_i\geq 0.\\\label{eq:Jensen}
 \frac{1}{\pi}\int_{r_0}^{r}\frac{d\tau}{\tau}\int_{\bD_\tau}\hess v=\frac{1}{2\pi}\int_{0}^{2\pi}v(re^{i\theta})d\theta-\frac{1}{2\pi}\int_{0}^{2\pi}v(r_0e^{i\theta})d\theta
	\end{align}
	for all functions $v$ so that $\hess v$ exists as measures (\emph{e.g.} $v$ is the difference of two subharmonic functions). \qed
\end{lem}
\subsection{A criteria for big Picard   theorem} 
\begin{dfn}[Finsler metric]\label{def:Finsler}
	Let \(E\) be a  holomorphic vector bundle on a complex manifold $X$. A \emph{Finsler metric} on \(E\) is a real non-negative  \emph{continuous}  function \(h:E\to  \mathclose[ 0,+\infty\mathopen[ \) such
	that 
	\[h(av) = |a|h(v)\]
	for any \(a\in \bC \) and \(v\in  E\).  	The metric $h$ is \emph{degenerate} at a point \(x\in X \)   if \(h(v)=0\) for some
	nonzero \(v\in E_{x} \), and the set of such degenerate points  is denoted by  \(\Delta_{h} \).
\end{dfn}
We shall mention that our definition is a bit different from that in \cite[Chapter 2, \S 3]{Kob98}, which requires \emph{convexity}, and the Finsler metric therein    can be upper-semi continuous. 

Let us now state and prove the main result in this section.
\begin{thm}[Criteria for big Picard theorem]\label{thm:Big Picard}
	Let $X$ be a projective manifold   and let $D$ be a simple normal crossing divisor on $X$. Let $h$ be a (possibly degenerate) Finsler metric of $T_X(-\log D)$.  Assume that $f:\bD^*\to X\setminus D$  is a  holomorphic map such that $  |f'(t)|_h^2\not\equiv 0$, and 
\begin{align}\label{eq:condition}
	\frac{1}{\pi}\hess \log |f'(t)|_h^2\geq f^*\omega
\end{align}
for some smooth K\"ahler metric $\omega$ on $X$. Then $f$ extends to a holomorphic map $\bar{f}:\mathbb{D}\to X$.
\end{thm}
The proof is an application of logarithmic derivative lemma, which is inspired by \cite[\S 4]{Dem97b} and \cite[Lemma 4.7.1]{NW14}.
\begin{proof}[Proof of \cref{thm:Big Picard}]
	We take a finite affine covering $\{U_\alpha\}_{\alpha\in I}$ of $X$ and rational holomorphic functions
	$(x_{\alpha1},\ldots,x_{\alpha n})$ on $U_\alpha$ so that
	\begin{align*}
	dx_{\alpha1}\wedge\cdots\wedge dx_{\alpha n}\neq 0 \ \mbox{ on } U_\alpha\\
	D\cap U_\alpha=(x_{\alpha,s(\alpha)+1} \cdots  x_{\alpha n}=0)
	\end{align*}
	Hence 
	\begin{align}\label{eq:basis}
(e_{\alpha 1},\ldots,e_{\alpha n}):=	(\frac{\d }{\d x_{\alpha 1}},\ldots,\frac{\d }{\d x_{\alpha s(\alpha)}},x_{\alpha, s(\alpha)+1}\frac{\d }{\d x_{\alpha, s(\alpha)+1}},\ldots,x_{\alpha n}\frac{\d}{\d x_{\alpha n}})
	\end{align} 
	is a basis for $T_X(-\log D)|_{U_\alpha}$. Write $$ (f_{\alpha1}(t),\ldots,f_{\alpha n}(t)):=(x_{\alpha 1}\circ f,\ldots,x_{\alpha n}\circ f).$$   With respect to the trivialization of $T_X(-\log D)$ induced by the basis \eqref{eq:basis},   $f'(t)$ can be written as
	$$f'(t)= f_{\alpha1}'(t)e_{\alpha 1}+\cdots+f_{\alpha s(\alpha)}'(t)e_{\alpha s(\alpha)}+(\log f_{\alpha,s(\alpha)+1})'(t)e_{\alpha, s(\alpha)+1}+\cdots+(\log f_{\alpha  n})'(t)e_{\alpha n}.$$
	Let $\{\rho_\alpha \}_{\alpha\in I}$ be a partition of unity subordinated to $\{U_\alpha\}_{\alpha\in I}$.   Since $h$ is Finsler metric for $T_X(-\log D)$ which is continuous and locally bounded from above by \cref{def:Finsler}, and $ I$ is a finite set,    there is a constant $C>0$ so that
\begin{align}\label{eq:unity}
	\rho_\alpha\circ f\cdot|f'(t)|_h^2\leq C \Big( \sum_{j=1}^{s(\alpha)}\rho_\alpha\circ f\cdot|f_{\alpha j}'(t)|^2+\sum_{i=s(\alpha)+1}^{n}|(\log f_{\alpha i})'(t)|^2 \Big) \quad \forall t\in \mathbb{D}^*
	\end{align}
	for any $\alpha$. Hence
	\begin{align}\nonumber
	T_{f,\omega}(r)&:=\int_{r_0}^{r}\frac{d\tau}{\tau}\int_{\bD_\tau}f^*\omega\stackrel{\eqref{eq:condition}}{\leq} \int_{r_0}^{r}\frac{d\tau}{\tau}\int_{\bD_\tau}		\frac{1}{\pi}\hess \log |f'|_h^2\\\nonumber
	&\stackrel{\eqref{eq:Jensen}}{\leq}  \frac{1}{2\pi}\int_{0}^{2\pi} \log|f'(re^{i\theta})|_h d\theta-\frac{1}{2\pi}\int_{0}^{2\pi} \log|f'(r_0e^{i\theta})|_h d\theta\\\nonumber
	&\leq \frac{1}{2\pi}\int_{0}^{2\pi} \log^+|f'(re^{i\theta})|_h d\theta+O(1)= \frac{1}{2\pi}\int_{0}^{2\pi} \log^+\sum_{\alpha}|\rho_\alpha\circ f\cdot f'(re^{i\theta})|_h d\theta+O(1)\\\nonumber
	&\stackrel{\eqref{eq:concave}}{\leq} \sum_{\alpha}\frac{1}{2\pi}\int_{0}^{2\pi} \log^+|\rho_\alpha\circ f\cdot f'(re^{i\theta})|_h d\theta+O(1)\\\nonumber
	&\stackrel{\eqref{eq:unity}}{\leq} \sum_{\alpha}\sum_{i=s(\alpha)+1}^n \frac{1}{2\pi}\int_{0}^{2\pi} \log^+|(\log f_{\alpha i})'(re^{i\theta})| d\theta\\\nonumber
	&\quad +\sum_{\alpha}\sum_{j=1}^{s(\alpha)}\frac{1}{2\pi}\int_{0}^{2\pi}\log ^+|\rho_\alpha\circ f\cdot f_{\alpha j}'(re^{i \theta})|d\theta +O(1)\\\nonumber
	&\stackrel{\cref{lem:Noguchi}}{\leq} C_1\sum_{\alpha}\sum_{i=s(\alpha)+1}^{n}\big(\log^+  T_{f_{\alpha i},\omega_{FS}}(r) +\log r\big)\\\nonumber
	&\quad +\sum_{\alpha}\sum_{j=1}^{s(\alpha)}\frac{1}{2\pi}\int_{0}^{2\pi}\log ^+|\rho_\alpha\circ f\cdot f_{\alpha j}'(re^{i \theta})|d\theta+O(1)\quad \lVert \\\label{eq:final}
	&\stackrel{\cref{lem:decrease}}{\leq} C_2(\log^+T_{f,\omega}(r)+ \log r)+\sum_{\alpha}\sum_{j=1}^{s(\alpha)}\frac{1}{2\pi}\int_{0}^{2\pi}\log ^+|\rho_\alpha\circ f\cdot f_{\alpha j}'(re^{i \theta})|d\theta+O(1)\quad \lVert.
	\end{align}
	where $C_1$ and $C_2$ are two positive constants which do not depend on $r$.
	\begin{claim}\label{claim}
		For any $\alpha\in I$ and any $j\in \{1,\ldots,s(\alpha)\}$, one has
	\begin{align}\label{eq:new}
\frac{1}{2\pi}\int_{0}^{2\pi}\log ^+|\rho_\alpha\circ f\cdot f_{\alpha j}'(re^{i \theta})|d\theta\leq C_3(\log^+T_{f,\omega}(r)+ \log r)+O(1)\quad \lVert
	\end{align} 
	for	positive constant $C_3$ which does not depend on $r$.
	\end{claim} 
	\begin{proof}[Proof of \cref{claim}]
	The proof of the claim is borrowed from \cite[eq.(4.7.2)]{NW14}. Pick $C>0$ so that $\rho_\alpha\sn dx_{\alpha j}\wedge d\bar{x}_{\alpha j}\leq C\omega$. Write $f^*\omega:=\sn B(t)dt\wedge d\bar{t}$. Then
		\begin{align*}
		&\frac{1}{2\pi}\int_{0}^{2\pi}\log ^+|\rho_\alpha\circ f\cdot f_{\alpha j}'(re^{i \theta})|d\theta = \frac{1}{4\pi}\int_{0}^{2\pi}\log ^+(|\rho^2_\alpha\circ f |\cdot|f_{\alpha j}'(re^{i \theta})|^2)d\theta\\
		&\leq \frac{1}{4\pi}\int_{0}^{2\pi}\log ^+B(re^{i\theta})d\theta+O(1) \leq \frac{1}{4\pi}\int_{0}^{2\pi}\log (1+B(re^{i\theta}))d\theta+O(1)\\
		&\leq \frac{1}{2}\log (1+\frac{1}{2\pi}\int_{0}^{2\pi}   B(re^{i\theta})d\theta)+O(1) 
	 = \frac{1}{2}\log (1+\frac{1}{2\pi r}\frac{d}{d r}\int_{\mathbb{D}_r}   rBdrd\theta)+O(1)\\
		&= \frac{1}{2}\log (1+\frac{1}{2\pi r}\frac{d}{d r} \int_{\bD_r}f^*\omega)+O(1)\\
		&\stackrel{\cref{lem:Borel}}{\leq}  \frac{1}{2}\log (1+\frac{1}{2\pi r} (\int_{\bD_r}f^*\omega)^{1+\delta})+O(1)\quad \lVert\\
		&=  \frac{1}{2}\log (1+\frac{r^\delta}{2\pi } (\frac{d}{d r} T_{f,\omega}(r))^{1+\delta})+O(1)\quad \lVert\\ 
		&\stackrel{\cref{lem:Borel}}{\leq}   \frac{1}{2}\log (1+\frac{r^\delta}{2\pi } ( T_{f,\omega}(r))^{(1+\delta)^2})+O(1)\quad \lVert\\
		&\leq 4\log^+T_{f,\omega}(r)+\delta\log r+O(1)\quad \lVert.
		\end{align*}
		Here we pick $0<\delta<1$ to apply \cref{lem:Borel}. 
		The claim is proved.
	\end{proof}
Putting \eqref{eq:new} to \eqref{eq:final},  one obtains
	$$
	T_{f,\omega}(r)\leq C(\log^+T_{f,\omega}(r)+ \log r)+O(1)\quad \lVert 
	$$
	for some positive constant $C$. 
	Hence $T_{f,\omega}(r)=O(\log r)$.  We  apply \cref{lem:criteria} to conclude that $f$ extends to the $\infty$. 
\end{proof}

\section{The negatively curved metric via Viehweg-Zuo Higgs bundles}
In \cite{VZ02,VZ03}, Viehweg-Zuo  introduced a special type of Higgs bundles over the bases of smooth families of polarized manifolds with semi-ample canonical sheaves to study the hyperbolicity of such bases. It was later developed in \cite{PTW18}. For any quasi-projective manifold $V$ endowed with a Viehweg-Zuo Higgs bundle,   we construct in \cite{Den18,Den18b} a generically positively definite Finsler metric over $V$  whose  holomorphic sectional curvature is bounded from above by a negative constant. In this section we will refine curvature estimate in \cite{Den18} to prove   big Picard  theorem for such quasi-projective  manifold $V$. 
%\subsection{Notations on Finsler metrics}

%\begin{dfn}\label{HSC}	
%	Let \(X\) be a complex manifold endowed with  a Finsler metric \(h\). 
%	\begin{thmlist}
%		\item \label{def:sectional}For any \(x\in X \), and \(v\in  T_{X,x} \), let \([v]\) denote the complex line spanned by \(v\). We define the holomorphic sectional curvature \(K_{F,[v]}\) in the direction of \([v] \) by 
%		\[
%		K_{h,[v]}:= \sup K_{\gamma^*h^2}(0)
%		\]
%		where the supremum is taken over all \(\gamma:\bD\rightarrow X \) such that \(\gamma(0)=x\) and \([v]\)
%		is tangent to \(\gamma'(0)\).
%		\item \label{negatively curved}	We say that   $h$  is \emph{negatively curved} if  there is a
	%	positive constant \(c\) such that \(K_{h,[v]}\leqslant -c\) for all \(v\in T_{X,x} \) for which \(h(v)>0\).
%		\item \label{degenerate}	A point \(x\in X \) is   a \emph{degeneracy point} of \(h\) if \(h(v)=0\) for some
	%	nonzero \(v\in T_{X,x} \), and the set of such points  is denoted by  \(\Delta_{h} \).%, and  a \emph{total degeneracy point} of \(F\) if \(F(v)=0\) for all nonzero \(v\in \ts_{X,x} \).  
		%	The set of  degeneracy points is called  degeneracy set of \(X\),  and denoted by  \(\Delta_{F} \).
		%	\item We say \(X\) is \emph{hyperbolic modulo} a subset \(\Delta\subset X \) if  $\kappa_X$  is positive definite outside $\Delta$.  %for every pair of distinct points \((p,q)\) of \(X\) we have
		%	\(d_X(p,q)\) unless both are contained in \(\Delta \). 
%	\end{thmlist}
%\end{dfn}
%A complex manifold $X$ is called \emph{quasi-compact} if it is a Zariski open set of a compact complex manifold $\overline{X}$.
 
\subsection{Viehweg-Zuo Higgs bundles and their proper metrics}\label{sec:VZ}
The definition for Viehweg-Zuo Higgs bundles we present below follows from the formulation in \cite{VZ02,VZ03} and \cite{PTW18}. %We mainly follow   the approach of the construction of VZ Higgs bundles in \cite{VZ02,PTW18}, whereas \cref{iterate,injection}    follow  from \cite[Propositions 2.7 \& 2.11]{PTW18}.   Let us  stress here that we    pursue the  simplified spirit of  \cite{PTW18},  which relaxes the monodromy condition in \cite{VZ03}.
\begin{dfn}[Abstract Viehweg-Zuo Higgs bundles]\label{def:VZ}
	Let $V$ be a quasi-projective manifold, and let $Y\supset V$ be a projective compactification of $V$ with the boundary $D:=Y\setminus V$ simple normal crossing. A \emph{Viehweg-Zuo Higgs bundle}  over $(Y,D)$ (or say over $V$   abusively) is a logarithmic  Higgs bundle $(\tilde{\sE},\tilde{\theta})$  over $Y$  consisting of the following data:
	\begin{thmlist}
		\item  a divisor $S$ on $Y$ so that $D+S$ is simple normal crossing,
		\item \label{VZ big}  a big and nef line bundle $L$ over $Y$ with $\mathbf{B}_+(L)\subset D\cup S $, 
		\item   a Higgs bundle  $( {\sE}, {\theta}):=\big(\bigoplus_{q=0}^{n}E^{n-q,q},\bigoplus_{q=0}^{n}\theta_{n-q,q}\big)$ induced by the lower canonical extension of a polarized VHS defined over $Y\setminus (D\cup S)$, 
		\item a sub-Higgs sheaf $(\sF,\eta)\subset (\tilde{\sE},\tilde{\theta})$,
	\end{thmlist}
	which	satisfy the following properties.
	\begin{enumerate}[leftmargin=0.5cm]
		\item The Higgs bundle $(\tilde{\sE},\tilde{\theta}):=(L^{-1}\otimes {\sE},\vvmathbb{1}\otimes {\theta})$. In  particular, $\tilde{\theta}:\tilde{\sE}\to \tilde{\sE}\otimes \Omega_{Y}\big(\log (D+S)\big)$, and $\tilde{\theta}\wedge\tilde{\theta}=0$.
		\item The sub-Higgs sheaf $(\sF,\eta)$ has log poles only on the boundary $D$, that is, $\eta:\sF\to \sF\otimes\Omega_{Y}(\log D)$.
		\item \label{contain trivial}Write $\tilde{\sE}_k:=L^{-1}\otimes E^{n-k,k}$, and denote by $\sF_k:=\tilde{\sE}_k\cap \sF$. Then the first stage $\sF_0$ of $\sF$ is an \emph{effective line bundle}. In other words, there exists a non-trivial morphism $\cO_Y\to \sF_0$.
	\end{enumerate}
\end{dfn}
As shown in \cite{VZ02}, by iterating $\eta$ for $k$-times, we obtain
$$
\sF_0\xrightarrow{\overbrace{\eta\circ\cdots\circ \eta}^{k\, \text{times}}} \sF_k\otimes \big(\Omega_Y(\log D)\big)^{\otimes k}.
$$
Since $\eta\wedge\eta =0$, the above morphism factors through $ \sF_k\otimes {\rm Sym}^k\Omega_Y(\log D)$, and by \eqref{contain trivial} one thus obtains 
$$
\cO_Y\to \sF_0\to \sF_k\otimes {\rm Sym}^k\Omega_Y(\log D)\to L^{-1}\otimes E^{n-k,k}\otimes {\rm Sym}^k\Omega_Y(\log D).
$$
Equivalently, we have a morphism
\begin{align}\label{iterated Kodaira2}
\tau_k:  {\rm Sym}^k T_Y(-\log D)\rightarrow L^{-1}\otimes E^{n-k,k}.
\end{align}
It was proven in \cite[Corollary 4.5]{VZ02} that $\tau_1$ is always non-trivial.
In \cite{Den18b} we prove that $\tau_1:T_Y(-\log D)\to L^{-1}\otimes E^{n-1,1}$ in \eqref{iterated Kodaira2}  is  generically injective.  

We will follow \cite{PTW18} to give some \enquote{proper} metric on $\tilde{\sE}=\oplus_{k=0}^{n}L^{-1}\otimes E^{n-k,k}$.
Write the simple normal crossing divisor \(D=D_1+\cdots+D_k \) and \(S=S_1+\cdots+S_\ell \). Let \(f_{D_i}\in H^0\big(Y,\cO_Y(D_i)\big)\) and \(f_{S_i}\in H^0\big(Y,\cO_Y(S_i)\big)\) be the canonical section defining \(D_i\) and \(S_i\). We fix  smooth hermitian metrics  \(g_{D_i}\) and \(g_{S_i}\) on \(\cO_Y(D_i)\) and \(\cO_Y(S_i)\). After rescaling \(g_{D_i} \) and \(g_{S_j}\), we assume that  $|f_{D_i}|_{g_{D_i}}< 1$ and $|f_{S_j}|_{g_{S_j}}< 1$ for  $i=1,\ldots,k$ and $j=1,\ldots,\ell$. Set
\[
r_{D}:=\prod_{i=1}^{k}(- \log |f_{D_i}|^2_{g_{D_i}}), \quad r_{S}:=\prod_{j=1}^{\ell}(- \log | f_{S_j}|^2_{g_{S_j}}).
\]
Let $g$ be a singular hermitian metric with analytic singularities of the  big and nef  line bundle \(L\) such that \(g\) is smooth on $Y\setminus \mathbf{B}_+(L)\supset Y\setminus D\cup S$,  and the curvature current 
$
\sqrt{-1}\Theta_{g}(L)\geqslant \omega
$
for some smooth K\"ahler form \(\omega \) on $Y$. For \(\alpha\in \mathbb{N}\),  define
\[
h_L:=g\cdot (r_D\cdot r_S)^\alpha
\]
The following proposition is a slight variant of \cite[Lemma 3.1, Corollary 3.4]{PTW18}.\noindent
\begin{prop}[\!\protect{\cite{PTW18}}]\label{singular metric}
	When \(\alpha\gg 0 \), after rescaling \(f_{D_i} \) and \(f_{S_i}\), there exists a continuous, positively definite hermitian form \(\omega_\alpha\) on \(T_{Y}(-\log {D}) \) such that
	\begin{thmlist} 
		\item\label{estimate}  over $V_0:=Y\setminus D\cup S$,  the curvature form
		\[
		\sqrt{-1}\Theta_{h_L}(L)_{\upharpoonright V_0} \geqslant r_D^{-2}\cdot \omega_{\alpha\upharpoonright V_0},\quad 	\sqrt{-1}\Theta_{h_L}(L)\geq \omega
		\]
		where $\omega$ is a smooth K\"ahler metric on $Y$.
		\item \label{bounded} The singular hermitian metric \(h:=h_L^{-1}\otimes h_{\rm hod} \) on \(\tilde{\sE}=L^{-1}\otimes \sE\) is locally bounded  on \(Y\), and smooth outside \((D+S)  \), where $h_{\rm hod}$ is the Hodge metric for the Hodge bundle $\sE$. Moreover, \(h\) is degenerate  on \(D+S \). 
		\item \label{new bound} The singular hermitian metric  \(r_D^2h \) on  \(L^{-1}\otimes \sE\) is also locally bounded  on \(Y\) and is degenerate  on \(D+S \).  \qed
	\end{thmlist}
\end{prop} 
Hence by \cref{def:Finsler}, $h$ and $r_D^2h$ are both Finsler metrics on $\tilde{\sE}$.
 \subsection{Construction of negatively curved Finsler metric}  
We adopt the same notations as \cref{sec:VZ} throughout this subsection. Assume that the log manifold $(Y,D)$ is endowed with a Viehweg-Zuo Higgs bundle. In \cite[\S 3.4]{Den18} we construct Finsler metrics $F_1,\ldots,F_n$ on $T_Y(-\log D)$ as follows.   % construct such the Finsler metric on \(\ts_Y(-\log D) \) via the VZ Higgs bundles $(\sE,\theta)= \big(\bigoplus_{q=0}^{n}E^{n-q,q},\theta_{n-q,q}\big)$  in \cref{Higgs bundle}. 
By \eqref{iterated Kodaira2},  for each \(k=1,\ldots,n \), there exists
\begin{eqnarray}\label{contain}
\tau_k: {\rm Sym}^k  T_{Y}(-\log D) \rightarrow  L^{-1}\otimes E^{n-k,k}.
\end{eqnarray}
Then it follows from \Cref{bounded} that the Finsler metric \(h \) on \(  L^{-1}\otimes E^{n-k,k} \) induces  a Finsler metric \(F_k\) on \( T_Y(-\log D) \) defined as follows: for any \(e\in   T_{Y}(-\log D)_y \),
\begin{eqnarray}\label{k-Finsler}
F_k(e):=  h\big(\tau_k(e^{\otimes k})\big)^{\frac{1}{k}}
\end{eqnarray}
Let $C\subset \bC$ be any open set of $\bC$.  For any \(\gamma:C\rightarrow  V \), one has
\begin{align} \label{eq:differential}
d\gamma: T_{C}\to \gamma^* T_V\hookrightarrow \gamma^* T_{Y}(-\log D).
\end{align}
 We	denote by \(\d_t:=\frac{\d}{\d t} \)  the canonical vector fields in \( C\subset \bC\),   \(\bar{\d}_t:=\frac{\d}{\d \bar{t}}\) its conjugate.  
The Finsler metric \(F_k\) induces a continuous Hermitian pseudo-metric on \(C \), defined by
\begin{eqnarray}\label{seminorm}
\gamma^*F_k^2=\sqrt{-1}G_k(t) dt\wedge d\bar{t}.
\end{eqnarray}
Hence $G_k(t)=|\tau_k\big(d\gamma(\d_t)^{\otimes k}\big)|^{\frac{2}{k}}_{h}$, where $\tau_k$ is defined in \eqref{iterated Kodaira2}.
 The reader might worried that all $G_k(t)$ will be identically equal to zero. In \cite[Theorem C]{Den18b}, we prove that \enquote*{generically} this cannot happen.
\begin{thm}[\protecting{\cite{Den18b}}]\label{Deng}
	There is a dense Zariski open set $V^\circ\subset V_0=Y\setminus(D+S)$ of $V^\circ$ so that $\tau_1:T_Y(-\log D)|_{V^\circ}\to L^{-1}\otimes E^{n-1,1}|_{V^\circ}$ is injective. \qed
\end{thm}
We now fix any  $\gamma:C\to V$ with  $\gamma(C)\cap V^\circ\neq \varnothing$.  By \Cref{bounded}, the metric $h$ for $L^{-1}\otimes \sE$ is smooth and positively definite over $V_0$. It then follows from  \cref{Deng} that $G_1(t)\not\equiv 0$.  Let $C^\circ\subset C$ be an (non-empty) open set whose  complement $C\setminus C^\circ$ is a \emph{discrete set} so that
\begin{itemize}
	\item  The image $\gamma(C^\circ)\subset V^\circ$.
	\item  For every $k=1,\ldots,n$,  either $G_k(t)\equiv 0$ on $C^\circ$ or $G_k(t)>0$ for any $t\in C^\circ$.
	\item $\gamma'(t)\neq0$ for any $t\in C^\circ$.
\end{itemize}
By the definition of $G_k(t)$, if $G_k(t)\equiv 0$ for some $k>1$, then $\tau_k(\d_t^{\otimes k})\equiv 0 $ where $\tau_k$ is defined in \eqref{iterated Kodaira2}. Note that one has $\tau_{k+1}(\d_t^{\otimes (k+1)})=\tilde{\theta}\big(\tau_k(\d_t^{\otimes k})\big)(\d_t)$,  where $\tilde{\theta}:L^{-1}\otimes \sE\to L^{-1}\otimes \sE\otimes \Omega_{Y}(\log (D+S))$ is defined in \cref{def:VZ}.  We thus conclude that $G_{k+1}(t)\equiv 0$. Hence it exists   $1\leq m\leq n$ so that the set $\{k\mid G_k(t)> 0 \mbox{ over }C^\circ\}=\{1,\ldots,m\}$, and $G_{\ell}(t)\equiv 0$ for all $\ell=m+1,\ldots,n$.  From now on, \emph{all the computations} are made over $C^\circ$.

In \cite{Den18} we proved the following curvature formula.
\begin{thm}[\protecting{ \cite[Proposition 3.12]{Den18}}]\label{thm:curvature}
	For $k=1,\ldots,m$, over $C^\circ$ one has
	\begin{align}  \label{eq:k=1}
	\frac{	\d^2\log G_1}{\d t\d \bar{t}}  	&\geq\Theta_{L,h_L}(\d_t,\bar{\d}_t) -  \frac{G_2^2}{G_1}    \quad  &\mbox{ if  }\ k=1, \\\label{eq:k>2}
	\frac{	\d^2\log G_k}{\d t\d \bar{t}}  	&\geq  \frac{1}{k}\Big( \Theta_{L,h_L}(\d_t, \bar{\d}_t)+  \frac{G_k^k}{G^{k-1}_{k-1}}  - \frac{G^{k+1}_{k+1}}{G^k_k}    \Big)  \quad  &\mbox{ if  }\ k>1. 
	\end{align} 
Here we make the convention that $G_{n+1}\equiv 0$ and $\frac{0}{0}=0$. We  also write $\d_t$ (resp. $\bar{\d}_t$) for $d\gamma(\d_t)$ (resp. ${d\gamma(\bar{\d}_t)}$) abusively, where $d\gamma$ is defined in \eqref{eq:differential}.   \qed
\end{thm}
Let us mention that  in \cite[eq. (3.3.58)]{Den18} we drop the term 
 $ \Theta_{L,h_L}(\d_t, \bar{\d}_t)$ in \eqref{eq:k>2}, though it can be easily seen from the proof of \cite[Lemma 3.9]{Den18}.  As we will see below, such a term is crucial in deriving the new curvature estimate. 

By \cref{thm:curvature} we see that the curvature of $F_k$ is   not the desired type \eqref{eq:condition} for applying  the criteria for big Picard theorem in  \cref{thm:Big Picard}.  
In \cite[\S 3.4]{Den18}, following ideas by \cite{TY14,Sch17} we introduce a new Finsler metric $F$ on $T_Y(-\log D)$ by taking convex sum in the following form
\begin{align}
F:=\sqrt{\sum_{k=1}^{n}k\alpha_kF_k^2}.
\end{align} 
where \(\alpha_1,\ldots,\alpha_n\in \mathbb{R}^+ \) are some constants  which will be fixed later. By \cref{Deng}, the set of degenerate points of  $F$ defined in \cref{def:Finsler}, denoted by $\Delta_F$,  is contained in a proper Zariski closed subset $Y\setminus V^\circ$.   In \cite[Theorem 3.8]{Den18} we prove that the holomorphic sectional curvature of $F$ is bounded from above by a negative constant.  Let us now prove a new curvature formula for $F$ in this section.

For the above \(\gamma:C\rightarrow  V \) with $\gamma(C)\cap V^\circ\neq \varnothing$, we write 
 $$
 \gamma^*F^2=\sqrt{-1}H(t)dt\wedge d\bar{t}.
 $$
 Then 
\begin{align}\label{eq:H}
H(t)=\sum_{k=1}^{n}  {{k\alpha_k}G_k}(t), 
\end{align} where \(G_k\) is defined in \eqref{seminorm}.  Recall that for $k=1,\ldots,m$,  $G_k(t)> 0$ for $t\in C^\circ$.  

We first recall a computational lemma by Schumacher.
\begin{lem}[\protecting{\cite[Lemma 17]{Sch17}}]
		Let \(\alpha_j>0 \) and $G_j$ be positive real numbers for \(j=1,\ldots,n \). Then 
		\begin{align}\nonumber
		&\sum_{j=2}^{n}\Big(\alpha_{j}\frac{G_j^{j+1}}{G_{j-1}^{j-1}}-\alpha_{j-1}\frac{G_j^{j}}{G_{j-1}^{j-2}} \Big)\\\label{final}
		&\geqslant \frac{1}{2}\Bigg(-\frac{\alpha_1^3}{\alpha_2^2}G_1^2+\frac{\alpha_{n-1}^{n-1}}{\alpha_n^{n-2}}G_n^2+\sum_{j=2}^{n-1}\bigg(\frac{\alpha_{j-1}^{j-1}}{\alpha_j^{j-2}} -\frac{\alpha_{j}^{j+2}}{\alpha_{j+1}^{j+1}}\bigg)G_j^2 \Bigg)
		\end{align}   \qed
\end{lem}
Now we are ready to compute the curvature of the Finsler metric $F$ based on \cref{thm:curvature}.
\begin{thm}  Fix a smooth K\"ahler metric $\omega$ on $Y$.	There exist  \emph{universal} constants \(0<\alpha_1<\ldots<\alpha_n\) and \(\delta>0\), such that for any $\gamma:C\to V$ with $C$ an open set of $\bC$ and $\gamma(C)\cap V^\circ\neq \varnothing$, one has
\begin{align}\label{eq:estimatefinal}
\hess \log|\gamma'(t)|_{F}^2\geq \delta\gamma^*\omega 
\end{align} 
\end{thm}
\begin{proof}
	By \cref{Deng} and the assumption that $\gamma(C)\cap V^\circ\neq \varnothing$, $G_1(t)\not\equiv0$. 
%Write $c_i(t):=\sn\Theta_{L,h_L}(\d_t,\bar{\d}_t)/G_i(t)$.  
We first recall a  result in   \cite[Lemma 3.11]{Den18}, and we write its proof here for it is crucial in what follows.
\begin{claim}\label{claim2}
There is a universal constant $c_0>0$ (\emph{i.e.} it does not depend on $\gamma$) so that $ \Theta_{L,h_L}(\d_t,\bar{\d}_t)\geq c_0G_1(t)$ for all $t$.
\end{claim}
\begin{proof}[Proof of \Cref{claim2}]
	Indeed, by \Cref{estimate}, it suffice to prove that  
\begin{eqnarray}\label{strict positive}
\frac{|\d_t|^2_{\gamma^*(r_D^{-2}\cdot \omega_\alpha)}}{|  \tau_1(d\gamma(\d_t))|_{h}^2}\geqslant c_0
\end{eqnarray}
for some $c_0>0$, where \(\omega_\alpha \) is a positively definite Hermitian metric on \(T_{Y}(-\log D )\). 
Note that
\[
\frac{|\d_t|^2_{\gamma^*(r_D^{-2}\cdot \omega_\alpha)}}{|  \tau_1(d\gamma(\d_t))|_{h}^2}=\frac{|\d_t|^2_{\gamma^*(r_D^{-2}\cdot \omega_\alpha)}}{|\d_t|^2_{\gamma^*\tau_1^*h}}=\frac{|\d_t|^2_{\gamma^*( \omega_\alpha)}}{|\d_t|^2_{\gamma^*\tau_1^*(r_D^2\cdot h)}},
\]
where \(\tau_1^*( r_D^{2}\cdot h) \) is a Finsler metric (indeed continuous pseudo hermitian metric) on \(T_{Y}(-\log D )\) by \Cref{new bound}.
  Since \(Y\) is compact,    there exists a constant  \(c_0>0\)  such that
\[
\omega_\alpha\geqslant c_0\tau_1^*( r_D^{2}\cdot h).
\] 
Hence  \eqref{strict positive} holds for any $\gamma:C\to V$ with $\gamma(C)\cap V^\circ\neq \varnothing$. The claim is proved.
\end{proof}

By \cite[Lemma 8]{Sch12}, 
		\begin{eqnarray}\label{calculus}
	\sqrt{-1}\d\bar{\d}\log (\sum_{j=1}^{n}{j\alpha_j}G_j )\geqslant \frac{\sum_{j=1}^{n}j\alpha_jG_j\sqrt{-1}\d\bar{\d}\log G_j}{\sum_{i=1}^{n}j\alpha_jG_i}
	\end{eqnarray}
Putting \eqref{eq:k=1} and \eqref{eq:k>2} to \eqref{calculus}, and making the convention that $\frac{0}{0}=0$, we obtain
	\begin{align*}\nonumber
	 \frac{\d^2\log H(t)}{\d t\d \bar{t}}&\geq  \frac{1}{H}\Big(-\alpha_1G_2^2+ \sum_{k=2}^{n}\alpha_k\big(\frac{G_k^{k+1}}{G^{k-1}_{k-1}}  - \frac{G^{k+1}_{k+1}}{G^{k-1}_{k}}   \big)\Big)+\frac{\sum_{k=1}^{n}\alpha_kG_k}{H} \Theta_{L,h_L}(\d_t,\bar{\d}_{t})\\
	 %&\frac{\alpha_1G^2_1}{H}\bigg(c_1(t)- \Big(\frac{G_2}{G_1}\Big)^2\bigg)+\frac{1}{H}\sum_{j=2}^{n} \alpha_jG^2_j  \bigg(c_j(t)+\Big(\frac{G_j}{G_{j-1}} \Big)^{j-1}-\Big(\frac{G_{j+1}}{G_j} \Big)^{j+1}  \bigg)\\\nonumber
		&= \frac{1}{H}\bigg( \sum_{j=2}^{n}\Big(\alpha_j\frac{G_j^{j+1}}{G_{j-1}^{j-1}}-\alpha_{j-1}\frac{G_j^{j}}{G_{j-1}^{j-2}} \Big) \bigg)+\frac{\sum_{k=1}^{n}\alpha_kG_k}{H} \Theta_{L,h_L}(\d_t,\bar{\d}_{t})\\\nonumber
		&\stackrel{\eqref{final}}{\geq}   \frac{1}{H}\bigg(-\frac{1}{2}\frac{\alpha_1^3}{\alpha_2^2}G_1^2+\frac{1}{2}\sum_{j=2}^{n-1}\Big(\frac{\alpha_{j-1}^{j-1}}{\alpha_{j}^{j-2}}-\frac{\alpha_j^{j+2}}{\alpha_{j+1}^{j+1}} \Big)G_j^2+\frac{1}{2}\frac{\alpha_{n-1}^{n-1}}{\alpha_n^{n-2}}G_n^2  \bigg) \\\nonumber
		&+\frac{\sum_{k=1}^{n}\alpha_kG_k}{H} \Theta_{L,h_L}(\d_t,\bar{\d}_{t})\\
		&\stackrel{\Cref{claim2}}{\geq} \frac{1}{H}\bigg(\frac{\alpha_1}{2}(c_0-\frac{\alpha_1^2}{\alpha_2^2})G_1^2+\frac{1}{2}\sum_{j=2}^{n-1}\Big(\frac{\alpha_{j-1}^{j-1}}{\alpha_{j}^{j-2}}-\frac{\alpha_j^{j+2}}{\alpha_{j+1}^{j+1}} \Big)G_j^2+\frac{1}{2}\frac{\alpha_{n-1}^{n-1}}{\alpha_n^{n-2}}G_n^2  \bigg) \\\nonumber
		&+\frac{1}{H}(\frac{1}{2}\alpha_1G_1+\sum_{k=2}^{n}\alpha_kG_k) \Theta_{L,h_L}(\d_t,\bar{\d}_{t}) 
	\end{align*}
	One can take \(\alpha_1=1 \), and choose the further \(\alpha_j>\alpha_{j-1} \) inductively   so that 
\begin{align}\label{eq:alpha}
c_0-\frac{\alpha_1^2}{\alpha_2^2}>0, \quad \frac{\alpha_{j-1}^{j-1}}{\alpha_{j}^{j-2}}-\frac{\alpha_j^{j+2}}{\alpha_{j+1}^{j+1}} >0 \quad \forall \ j=2,\ldots,n-1. 
\end{align}
	Hence 
	$$
	 \frac{\d^2\log H(t)}{\d t\d \bar{t}}\geq   \frac{1}{H}(\frac{1}{2}\alpha_1G_1+\sum_{k=2}^{n}\alpha_kG_k) \Theta_{L,h_L}(\d_t,\bar{\d}_{t}) \stackrel{\eqref{eq:H}}{\geq}\frac{1}{n}\Theta_{L,h_L}(\d_t,\bar{\d}_{t}) 
	$$ 
	over $C^\circ$. 
 By \Cref{estimate}, this implies that 
\begin{align}\label{eq:curvature esti}
\hess \log |\gamma'|_{F}^2=  \hess \log H(t)\geq \frac{1}{n}\gamma^*\sn\Theta_{L,h_L}\geq \delta\gamma^*\omega
\end{align}
 over $C^\circ$ for some positive constant $\delta$, which does not depend on $\gamma$.  Since $|\gamma'(t)|_F^2$ is continuous and locally bounded from above over   $C$,  by the extension theorem of  subharmonic function, \eqref{eq:curvature esti}  holds over the whole $C$.   Since $c_0>0$ is a constant which does not depend on $\gamma$, so are $\alpha_1,\ldots,\alpha_n$ by  \eqref{eq:alpha}.  The theorem is thus proved.
\end{proof}
In summary of results in this subsection, we obtain the following theorem.
\begin{thm}\label{thm:uniform}
	Let $X$ be a projective manifold  and let $D$ be a simple normal crossing divisor on $X$. Assume that there is a Viehweg-Zuo Higgs bundle over $(X,D)$. Then there are a Finsler metric $h$ on $T_X(-\log D)$ which is positively definite on a dense Zariski open set $V^\circ$ of $V:=X\setminus D$, and a smooth K\"ahler form on $X$ such that   for any holomorphic map $\gamma:C\to V$  from any open subset $C$ of $ \bC$ with $\gamma(C)\cap V^\circ\neq \varnothing$, one has
$$
\hess \log |\gamma'|_h^2\geq \gamma^*\omega.
$$ 
In particular, by \cref{main1},	   for any holomorphic map   $f:\Delta^*\to X\setminus D$, with $f(\Delta^*)\cap V^\circ\neq \varnothing$,  it extends to a holomorphic map $\bar{f}:\Delta\to X$. \qed
\end{thm}
\section{Generalized Big Picard theorems}
We will apply \cref{thm:Big Picard,thm:uniform} to prove the big Picard theorem   for moduli of polarized manifolds with semi-ample canonical sheaf, and for bases of log smooth families of  Calabi-Yau pairs.
\subsection{Big Picard theorem for moduli of polarized manifolds}  
\begin{proof}[Proof of \cref{main}]
Let $Z$ be the Zariski closure of $\gamma(\Delta^*)$ in $Y$. Take an embedded desingularization of singularities $\mu:\tilde{Y}\to Y$ so that the strict transform  of $Z$, denoted by $\tilde{Z}$, is smooth. 
 Write $\tilde{Z}^\circ:= \tilde{Z}\cap \mu^{-1}(V)$, which is a dense Zariski open set of $\tilde{Z}$. We take the base change
 \begin{equation*}
 \begin{tikzcd}
 X^\circ=U\times_V\tilde{Z}^\circ \arrow[r] \arrow[d,"f_{X^\circ}"'] & U \arrow[d,"f_U"]\\
 \tilde{Z}^\circ \arrow[r,"\iota"] & V 
 \end{tikzcd}
 \end{equation*}
 Then  polarized family $(f_{X^\circ}:X^\circ\to \tilde{Z}^\circ,\iota^*\sL)\in \sP_h({\tilde{Z}^\circ})$. We denote by  $\varphi_{X^\circ}:\tilde{Z}^\circ\to P_h$   the moduli map  associated to $f_{X^\circ}$. Then $\varphi_{X^\circ}=\varphi_U\circ\iota$, which is generically finite. Hence $f_{X^\circ}$ is of maximal variation. By \cite{VZ02,PTW18}, after passing to a birational modification $\nu:W\to \tilde{Z}$,  there exists a Viehweg-Zuo Higgs bundle on   $W^\circ:=\nu^{-1}(\tilde{Z}^\circ)$. By \cref{thm:uniform}, there is a dense Zariski open set $W'\subset W^\circ$ so that any holomorphic map $\Delta^*\to W^\circ$ extends to $\Delta\to W$ provided that its image is not contained in $W\setminus W'$.  Since   $\gamma:\Delta^*\to Z$ is Zariski dense, it thus does not lie on the discriminant locus of the birational morphism $\mu|_{\tilde{Z}}\circ \nu:W\to Z$, and thus $\tilde{\gamma}=(\mu|_{\tilde{Z}}\circ \nu)^{-1}\circ\gamma:\Delta^*\to W$ exists with its image contained in $W^\circ$.  Moreover, $\tilde{\gamma}:\Delta^*\to W$   is also Zariski dense, and thus $\tilde{\gamma}(\Delta^*)\cap W'\neq \varnothing$. By \cref{thm:Big Picard}, $\tilde{\gamma}:\Delta^*\to W$ extends to a holormorphic map  $\overline{\tilde{\gamma}}:\Delta\to W$. The holomorphic map $\mu\circ \nu\circ\overline{\tilde{\gamma}}:\Delta\to Y$ is the desired extension of $\gamma:\Delta^*\to V$.  The theorem is proved.
\end{proof}
\begin{rem}\label{rem:LSZ}
	Based on the fundamental work \cite{VZ02,VZ03,PTW18}, in \cite{Den18b} we prove that the base $V$ in \cref{main} is both Brody hyperbolic and pseudo Kobayashi hyperbolic.  In \cite{LSZ19},  Lu-Sun-Zuo combine the original  approach in \cite{VZ03} with our Torelli type result \cref{Deng} to construct negatively curved pseudo hermitian metric on any algebraic curve in $V$, so that they can apply the celebrated work of Griffiths-King \cite{GK73} to prove the Borel hyperbolicity of $V$. 
\end{rem}
\subsection{Big Picard theorem for bases of log Calabi-Yau families}
In \cite{Den19}, we prove that for maximally varying, log smooth family of Calabi-Yau pairs, its base can be equipped with a Viehweg-Zuo Higgs bundle. Let us start with the following definition. 
\begin{dfn}[log smooth family]\label{def:log smooth}
	Let $X^\circ$ and $Y^\circ$ be   quasi-projective manifolds, and let $D^\circ=\sum_{i=1}^{m}a_iD^\circ_i$ be a Kawamata log terminal (klt for short) $\mathbb{Q}$-divisor on $X^\circ$ with simple normal crossing support. The  morphism $f^\circ:(X^\circ,D^\circ)\to Y^\circ$ is a \emph{log smooth family} if $f^\circ:X^\circ\to Y^\circ$ is a smooth projective morphism with connected fibers, and $D^\circ$ is \emph{relatively normal crossing over $Y^\circ$}, namely each stratum $D^\circ_{i_1}\cap \cdots\cap D^\circ_{i_k}$ of $D^\circ$ is dominant onto and smooth over $Y^\circ$ under $f^\circ$. 
\end{dfn}
Let us recall the main result in \cite{Den19}.
\begin{thm}[\protecting{\cite[Theorem A]{Den19}}] \label{thm:CYmain}
	Let $f^\circ:(X^\circ,D^\circ)\to Y^\circ$ be a log smooth family  over a quasi-projective manifold $Y^\circ$. Assume that  each fiber $(X_y,D_y):=(f^\circ)^{-1}(y)$ is a klt pair, and  $K_{X_y}+D_y\equiv_{\mathbb{Q}}0$. Assume that  the  logarithmic Kodaira-Spencer map $\rho:T_{Y^\circ}\to R^1f^\circ_*\big(T_{X^\circ/Y^\circ}(-\log D^\circ)\big)$ is generically injective.   Then after replacing $Y^\circ$ by a birational model, there exists a Viehweg-Zuo  Higgs bundle over some   $Y^\circ$. \qed
\end{thm}
By \cref{thm:CYmain}, one   can  perform the same proof as that of \cref{main} to conclude the following result. 
\begin{thm} \label{thm:CY}
	In the setting of \cref{thm:CYmain}, assume additionally that the logarithmic Kodaira-Spencer map 
		$$
		T_{Y^\circ,y}\rightarrow  H^1 \big(X_y,T_{X_y}(-\log D_y)\big)
		$$
		is injective for any $y\in Y^\circ$. Then for any projective compactification $Y$ of the base $Y^\circ$,  any holomorphic map $\gamma:\Delta^*\to Y^\circ$ extends into the origin. \qed
\end{thm}

\bibliographystyle{amsalpha}
\bibliography{biblio}

\end{document}